\documentclass[a4paper,12pt]{article}

\usepackage{amsmath ,amsfonts,  amssymb, amscd, amsthm} 
\usepackage{bbm}
\usepackage{mathrsfs}
\usepackage[font=small,skip=0pt]{caption}

\usepackage{csquotes, color}
\usepackage{colortbl}
\usepackage[numbers]{natbib}
\usepackage{verbatim}
\usepackage[toc,page]{appendix}
\usepackage{caption}
\usepackage{textcomp}
\usepackage[inline]{enumitem}
\usepackage{tcolorbox}
\usepackage{cancel}
\usepackage{relsize}
\usepackage{bm}

\usepackage{authblk}

\usepackage{algorithm}
\usepackage{algorithmic}
\usepackage{scrextend}
\addtokomafont{labelinglabel}{\sffamily}

\usepackage{tikz}
\usepackage{tikz-cd}
\usetikzlibrary{arrows.meta}
\usetikzlibrary{positioning}
\usetikzlibrary{shadows}
\usetikzlibrary{calc, intersections}
\usetikzlibrary{shapes.multipart}
\usetikzlibrary{patterns}

\newtheorem{theorem}{Theorem}

\newtheorem{proposition}{Proposition}

\newtheorem{definition}{Definition}
\newtheorem*{definition*}{Definition}

\newtheorem*{example*}{Example}

\newtheorem*{conditional*}{Conditional}

\newtheorem*{remark*}{Remark}

\tikzset{Myarrow/.style={very thin, arrows=-Latex}}

\definecolor{Gray}{gray}{0.85}
\definecolor{LightCyan}{rgb}{0.88,1,1}

\newcolumntype{a}{>{\columncolor{Gray}}c}
\newcolumntype{b}{>{\columncolor{Gray}}r}
\newcolumntype{d}{>{\columncolor{white}}c}

\newcommand{\scrL}{{\mathscr L}}

\usepackage{hyperref}
\hypersetup{
    colorlinks=true,
    citecolor=black,
    linkcolor=black,
    filecolor=cyan,      
    urlcolor=blue,
}
 
\providecommand{\keywords}[1]
{
  \small	
  \textbf{\textit{Keywords---}} #1
}

\title{
A Probabilistic Perspective 
on Feller, Pollard and the Complete Monotonicity 
of the Mittag-Leffler Function
}
\author{Nomvelo Karabo Sibisi \\{\small {\tt sbsnom005@myuct.ac.za}}}
\date{\today}

\begin{document}
\maketitle
\thispagestyle{empty}


\begin{abstract}
\noindent
The main contribution of this paper is the use  of probability theory to prove that
the three-parameter Mittag-Leffler 
function is the Laplace transform of a distribution   
and thus  completely monotone.
Pollard used contour integration 
to  prove the result in the one-parameter case. 
He  also cited personal communication by Feller of a 
discovery of the result by  ``methods of probability theory''.
Feller 
used the two-dimensional Laplace transform of a  bivariate distribution to derive the   result.
We 
pursue the theme of  probability theory
to explore complete monotonicity beyond the contribution  due to Feller.
Our approach involves an  interplay between  mixtures and convolutions of stable and gamma densities, 
together with a limit theorem  that leads to a novel integral representation of the 
three-parameter Mittag-Leffler function (also known as  the  Prabhakar function). 
\end{abstract}

\section{Introduction}
\label{sec:intro}

The  problem of  interest in this paper is the study of  
the complete monotonicity of the Mittag-Leffler function.
Complete monotonicity is an analytic property of functions.
Accordingly, Pollard~\cite{PollardML} used analytic methods to prove the property in the instance of the  Mittag-Leffler function.
Pollard   also cited personal communication by Feller of a discovery of the result by  ``methods of probability theory''.
However, Pollard's comment notwithstanding,
the published proof by Feller~\cite{Feller2}~(XIII.8) also may be regarded as more analytic  than probabilistic
(we discuss both approaches later in this section).
This prompted us to ask the following:  
\begin{enumerate}
\item What  might constitute  a  ``method of probability theory''  in proving an analytic property of a function,
at least in the context of proving that the Mittag-Leffler function is completely monotone?
\item What additional or complementary  insight, if any, might the method of probability theory offer 
relative to an analytic method? 
\end{enumerate}
The strategy of this paper  is simple --   assign appropriate probability distributions 
and use the sum and product rules of probability theory to explore analytic attributes of associated  functions.
Beyond  reproducing known analytic results due to Pollard and Feller, 
we discuss the  generalisation that  flows from  adopting such  reasoning.
We start with definitions of complete monotonicity and the  Mittag-Leffler function.


\subsection{Definitions}
\label{sec:definitions}

An infinitely differentiable function $\varphi(x)$ on $x>0$ is completely monotone  if its derivatives $\varphi^{(n)}(x)$ 
satisfy $(-1)^n\varphi^{(n)}(x)\ge0$, $n\ge 0$.
Bernstein's theorem 
states that $\varphi(x)$ is completely monotone iff it may be expressed as 
\begin{align}
\varphi(x) &= \int_0^\infty e^{-x t}\,dF(t) 
= \int_0^\infty e^{-x t} f(t)dt
\label{eq:LT} 
\end{align}
for  a non-decreasing distribution function $F(t)$ with density $f(t)$, {\it i.e.}\ $F(t)=\int_0^t f(u)du$.
The first integral in~(\ref{eq:LT}) is formally called the Laplace-Stieltjes transform of $F$ and  
the latter the (ordinary) Laplace transform  of  $f$.
For bounded $F(t)$, $\varphi(x)$ is defined on $x\ge0$.
Integrating~(\ref{eq:LT}) by parts in this case gives  $\varphi(x)$ in terms of 
the ordinary Laplace transform of $F$:
\begin{align}
   \varphi(x) &=  x \int_{0}^\infty e^{-xt} F(t)\,dt = \int_{0}^\infty e^{-t} F(t/x)\,dt
\end{align}
The (one-parameter) Mittag-Leffler function 
$E_\alpha(x)$  is defined by  the infinite series 
\begin{align}
E_\alpha(x) &= \sum_{k=0}^\infty \frac{x^k}{\Gamma(\alpha k+1)} \quad \alpha\ge0
\label{eq:ML}
\end{align}
For later reference, the Laplace transform of 
$E_\alpha(-\lambda x^\alpha)$ $(\lambda>0)$ is 
\begin{align}
\int_0^\infty e^{-sx} E_\alpha(-\lambda x^\alpha) \,dx  &= \frac{s^{\alpha-1}}{\lambda+s^\alpha} \qquad {\rm Re}(s)\ge0
\label{eq:LaplaceML}
\end{align}
We shall introduce the two and three-parameter  generalisations below. 
For now, we may  turn to the  problem of proving  the complete monotonicity of $E_\alpha(-x)$.
We discuss  the approaches due  to Pollard and Feller in turn before turning to our probabilistic  perspective.

\subsection{Pollard's Method}
\label{sec:Pollard}

In a 1948 paper, Pollard~\cite{PollardML}  led with the  opening remark:
\begin{quote}
``W.~Feller communicated to me his discovery -- by the methods of probability theory -- that if $0\le \alpha \le1$ 
the function $E_\alpha(-x)$ is completely monotonic for $x\ge0$. 
This means that it can be written in the form
\begin{align*}
E_\alpha(-x) &= \int_{0}^\infty e^{-xt} dP_\alpha(t) 
\end{align*}
where $P_\alpha(t)$ is nondecreasing and bounded.
In this note we shall prove this fact directly  and determine the function $P_\alpha(t)$ explicitly.'' \newline
 [we use $P_\alpha$ where Pollard used $F_\alpha$, which we reserve for another purpose]
\end{quote}
Having dispensed with  $E_0(-x)=1/(1+x)$ and $E_1(-x)=e^{-x}$ since ``there is nothing to be proved in these cases'',
Pollard used  a contour integral representation of $E_\alpha(-x)$:
\begin{align}
 E_\alpha(-x) &= \frac{1}{2\pi i}\oint_{C} \frac{s^{\alpha-1}e^s}{x+s^{\alpha}}\,ds  
                       = \frac{1}{2\pi i\alpha}\oint_{C^\prime} \frac{e^{z^{\frac{1}{\alpha}}}}{x+z}\,dz
\label{eq:MLcontour}
\end{align}
to prove that 
\begin{alignat}{3}
p_\alpha(t) &\equiv P_\alpha^{\,\prime}(t) = \frac{1}{\alpha}\, f_\alpha(t^{-1/\alpha})\, t^{-1/\alpha-1}  \qquad && 0<\alpha<1
\label{eq:Pollarddensity}
\intertext{where $f_\alpha(t)$ is defined by} 
e^{-s^\alpha}  &= \int_0^\infty e^{-s t} f_\alpha(t)\,dt  && 0<\alpha<1 
\label{eq:stable}
\end{alignat}
Pollard~\cite{Pollard} had earlier proved that $f_\alpha(t)>0$, so that
$p_\alpha(t)\ge0$, thereby completing his proof that  $E_\alpha(-x)$ is completely monotone for $0\le \alpha \le1$.
Pollard stopped at the point of deriving~(\ref{eq:Pollarddensity}), the density  $p_\alpha(t)\equiv P_\alpha^{\,\prime}(t)$.
As per initial task, we proceed to discuss $P_\alpha(t)$ explicitly.
We first recognise $f_\alpha(t)$ as  the density  of the stable distribution $F_\alpha$ on $[0,\infty)$
\begin{alignat}{3}
F_\alpha(t) &= \int_0^t f_\alpha(u)\,du  \qquad && 0<\alpha<1 
\label{eq:stableF}
\end{alignat}
with normalisation $F_\alpha(\infty)=1$.
In turn, the Pollard distribution $P_\alpha(t)$ is 
\begin{align}
P_\alpha(t) &= \int_0^t p_\alpha(u)\,du  = \frac{1}{\alpha} \int_0^t  f_\alpha(u^{-1/\alpha})\, u^{-1/\alpha-1} \, du 
\label{eq:PollardP1}
\end{align}
Janson~\cite{Janson} derived $P_\alpha(t)$  as a limiting distribution of 
a  P{\' o}lya urn scheme.
 $P_\alpha(t)$ is known as  the Mittag-Leffler distribution in the probabilistic literature (one of  two distributions 
bearing the same name as discussed later).

Setting $y=u^{-1/\alpha}$ in~(\ref{eq:PollardP1}) gives a simple relation between $P_\alpha$ and $F_\alpha$:
\begin{align}
P_\alpha(t)  &= \int_{t^{-1/\alpha}}^\infty  f_\alpha(y)\, dy 
 =  1- \int_0^{t^{-1/\alpha}} f_\alpha(y)\, dy   
  \equiv 1-  F_\alpha(t^{-1/\alpha}) 
\label{eq:PollardP}
\end{align}
This `duality' between the Mittag-Leffler and stable distributions 
 is key to the discussion that follows.
The Pollard result may accordingly be written in several equivalent forms: 
\begin{align}
E_\alpha(-x) 
&= \int_{0}^\infty e^{-xt} dP_\alpha(t) = \int_{0}^\infty e^{-t} P_\alpha(t/x)\, dt   \nonumber \\
{\rm or} \quad 
E_\alpha(-x^\alpha) &= \int_{0}^\infty e^{-t} P_\alpha(x^{-\alpha}t)\, dt
   =  \int_{0}^\infty e^{-t} (1-  F_\alpha(xt^{-1/\alpha}))\,dt 
\label{eq:LaplacePollardP}
\end{align}
Another  representation arising  from  change of variable in Pollard's original result is
\begin{align}
\alpha E_\alpha(-x^\alpha) &=  \int_0^\infty e^{-x^\alpha u} \,  f_\alpha(u^{-1/\alpha})\, u^{-1/\alpha-1} \, du  \nonumber \\
  &= x  \int_0^\infty e^{-t} \,  f_\alpha(xt^{-1/\alpha})\, t^{-1/\alpha-1} \, dt
\label{eq:LaplacePollardvariant}
\end{align}
Setting aside Pollard's contour integral proof, it is hard  to evaluate directly any of the equivalent integral representations   above
to demonstrate that they do indeed generate  $E_\alpha(-x), E_\alpha(-x^\alpha)$.
A method that may be convenient to  prove one representation effectively proves all other representations because they are  
interchangeable  ways of stating  the Pollard result.
In particular, Feller followed an  indirect  route to prove the representation~(\ref{eq:LaplacePollardP}), discussed next.  

\subsection{Feller's Method}
\label{sec:Feller}

In an illustration of the use of the two-dimensional Laplace transform, 
Feller~\cite{Feller2}(p453) considered  $1-F_\alpha(xt^{-1/\alpha})$ as a bivariate  distribution over $x>0,t>0$.
The Laplace transform over $x$, followed by that over  $t$ gives
\begin{align}
\int_{0}^\infty e^{-sx} (1-F_\alpha(xt^{-1/\alpha}))\,dx &= \frac{1}{s}-\frac{e^{-ts^\alpha}}{s}
\label{eq:FellerLaplace1} \\
\frac{1}{s}  \int_{0}^\infty e^{-\lambda t} \left(1-e^{-ts^\alpha}\right)dt &= \frac{1}{\lambda} \frac{s^{\alpha-1}}{\lambda+s^\alpha}
 \label{eq:FellerLaplace2}
\end{align}
By reference to~(\ref{eq:LaplaceML}), the right hand side of~(\ref{eq:FellerLaplace2}) is the 
 Laplace transform of $E_\alpha(-\lambda x^\alpha)/\lambda$.
Since the two-dimensional Laplace transform equivalently can  be evaluated first over $t$ then over $x$,
Feller concluded that 
\begin{align}
E_\alpha(-\lambda x^\alpha) &= \lambda \int_{0}^\infty e^{-\lambda t} (1-F_\alpha(xt^{-1/\alpha}))\,dt 
\label{eq:FellerPollard} 
\end{align}
which, for $\lambda=1$, is  the Pollard result in the form~(\ref{eq:LaplacePollardP}).
Feller's proof  is based on the interchange of the order of integration (Fubini's theorem) and the 
uniqueness of Laplace transforms.
We represent it  by the commutative diagram below, where $\scrL_{s\vert t}$ denotes
the one-dimensional Laplace transform of a bivariate source function at fixed $t$, 
to give 
a  bivariate function of $(s,t)$ where $s$ is the Laplace variable.
\begin{equation}
\begin{tikzpicture}[auto,scale=1.8, baseline=(current  bounding  box.center)]
\newcommand*{\size}{\scriptsize}%
\newcommand*{\gap}{.2ex}%
\newcommand*{\width}{3.0}%
\newcommand*{\height}{1.5}%

\node (P) at (0,0)  {$1-F_\alpha(xt^{-1/\alpha})$};
\node (Q) at ($(P)+(\width,0)$) {$\dfrac{1}{s}-\dfrac{e^{-ts^\alpha}}{s}$};
\node (B) at ($(P)-(0,\height)$) {$\dfrac{1}{\lambda}E_\alpha(-\lambda x^\alpha)$}; 
\node (C) at ($(B)+(\width,0)$) {$\dfrac{1}{\lambda}\dfrac{s^{\alpha-1}}{\lambda+s^\alpha}$};   
\draw[Myarrow] ([yshift =  \gap]P.east)  --  node[above] {\size $\scrL_{s\vert t}$} node[below]{\size easy}  ([yshift = \gap]Q.west) ;
\draw[Myarrow]([xshift  =  \gap]Q.south) --  node[left] {\size $\scrL_{\lambda\vert s}$} node[right] {\size easy} ([xshift =  \gap]C.north);
\draw[Myarrow] ([xshift =  \gap]P.south) -- node[left] {\size $\scrL_{\lambda\vert x}$}  node[right] {\size hard}
([xshift =  \gap]B.north);
\draw[Myarrow] ([yshift = +\gap]C.west) --  node[above] {\size $\scrL^{-1}_{x\vert \lambda}$} node[below]{\size easy}  ([yshift  = +\gap]B.east); 
\end{tikzpicture}
\label{eq:Fellerdiagram}
\end{equation}
The desired  proof  is the   ``hard'' direct path, which is equivalent to the  ``easy'' indirect path.
We will return to commutative diagram representation in a different context later in the paper
when we discuss the main theorem. 

Feller's concise proof uses  ``methods of probability theory'', as cited  by Pollard, only to the extent 
of choosing the  bivariate distribution as input to the   two-dimensional  Laplace transform.
Short of any  further insight, the methods by both  Pollard and Feller might  be described as more analytic  than probabilistic.

\subsection{Purpose and Scope of Paper}
\label{sec:purpose}
We assign appropriate  distribution guided by the  
 task of proving  that $E_\alpha(-x)$ is completely monotone.
We  first cast Feller's argument in such terms before proceeding to a more general probabilistic discussion.



The Mittag-Leffler function is of growing interest in probability theory and physics, 
with a diversity of applications, notably fractional calculus.
A comprehensive study of the  properties and applications of  the Mittag-Leffler function 
and its numerous generalisations is beyond the scope of this paper.
We  consciously restrict the scope  to  the theme of  complete monotonicity and  Mittag-Leffler functions, 
underpinned by probability theory. 

Other studies that explicitly discuss complete monotonicity and  Mittag-Leffler functions 
build upon  complex analytic approaches similar to Pollard's rather than
the probabilistic  underpinning discussed here
(de Oliviera et al.~\cite{Oliveira}, Mainardi and Garrappa~\cite{MainardiGarrappa}, 
G\'{o}rska et al.~\cite{Gorska}).
These papers comment on the fundamental  importance of the complete monotonicity  of Mittag-Leffler functions used in 
the modelling of physical phenomena, such as anomalous dielectric relaxation and viscoelasticity.



Finally, we are keenly aware that there are other views on the interpretation of ``methods of probability theory''.  
We comment on this before discussing our probabilistic approach in detail.

\subsection{Probabilistic Perspectives}
\label{sec:perspectives}

The phrase  `methods of probability theory' used by Pollard may suggest an 
 experiment  with random outcomes as a  fundamental  metaphor. 
As noted earlier,  $P_\alpha$ is derived  as a limiting distribution of 
a  P{\' o}lya urn scheme  in the probabilistic literature. 

Diversity  of approach is commonplace in  probability theory and mathematics more generally.
For example, in a  context of nonparametric  Bayesian analysis, 
Ferguson~\cite{Ferguson1} constructed the Dirichlet process based on 
the  gamma distribution as the fundamental probabilistic concept, 
without invoking a random experiment.
Blackwell and MacQueen~\cite{BlackwellMacQueen}  observed that the Ferguson approach
``involves a rather deep study of the gamma process'' as they proceeded to give  an alternate construction 
based on the metaphor of a  generalised P{\' o}lya urn scheme.
Adopting the one approach is not to deny or diminish the other,  but to bring attention to the diversity of thinking in  probability theory,
 even when the end result is the same mathematical  object.
We look upon this as healthy complementarity rather than undesirable contestation.
 

For our purpose, we have no need to invoke an underlying random experiment or indeed  an explicit random variable,
while not denying the latter  as an alternative probabilistic approach.
Hence, for example, we  shall continue to express the Laplace transform of a distribution as an explicit integral 
rather than as an expectation $\mathbb{E}\left[e^{-sX}\right]$ 
for  a random variable $X$.

\section{A Probabilistic Method}
\label{sec:probabilistic}

First, we note that the scale change $s\to t^{1/\alpha}s$ $(t>0)$ in~(\ref{eq:stable})  gives
\begin{align}
 e^{-t s^\alpha}  &= \int_0^\infty e^{-s x} f_\alpha(x\,t^{-1/\alpha}) t^{-1/\alpha}\,dx 
 \equiv \int_0^\infty e^{-s x} f_\alpha(x \vert t) \,dx
\label{eq:stablescaled}
\end{align}
where $f_\alpha(x\vert t)\equiv f_\alpha(x\,t^{-1/\alpha}) t^{-1/\alpha}$ is 
the stable density conditioned on the scale parameter $t$,
with $f_\alpha(x) \equiv f_\alpha(x\vert 1)$.
Correspondingly, the  stable distribution conditioned  on $t$ is
\begin{align}
F_\alpha(x\vert t) &= \int_0^x\, f_\alpha(u\vert t)\,du
     =  \int_0^{x t^{-1/\alpha}}  f_\alpha(u) \,du  \equiv F_\alpha(x t^{-1/\alpha})
\label{eq:stablescaleddistribution}
\end{align}
with Laplace transform 
$ e^{-t s^\alpha}/s$.

We then assign a distribution $G(t)$  to the scale parameter $t$ of 
$F_\alpha(x\vert t)$.
Then,  by the sum and  product rules of probability theory, the  unconditional  or marginal distribution  $M_\alpha(x)$  over $x$ is
\begin{align}
M_\alpha(x) &= \int_0^\infty F_\alpha(x\vert t) dG(t)
\label{eq:marginaldistribution}
\intertext{with Laplace transform}
\int_0^\infty e^{-s x} M_\alpha(x) \,dx  
       &= \frac{1}{s} \int_0^\infty e^{-ts^\alpha}\, dG(t)
\label{eq:marginaldistributionLaplace}
\end{align}
$M_\alpha$ 
is also referred to as  a  mixture distribution, 
arising from randomising or mixing the parameter $t$ in $F_\alpha(x\vert t)$  with $G(t)$.
This has the same import as saying that we assign a prior distribution  $G(t)$ on $t$ 
and we shall continue to use the latter language.

 $G$ may depend on one or more parameters. 
A notable example is the gamma distribution $G(\mu,\lambda)$ with shape and scale parameters  $\mu>0, \lambda>0$ respectively:
\begin{align}
dG(t\vert \mu,\lambda) 
 &=\dfrac{\lambda^\mu}{\Gamma(\mu)}\,t^{\mu-1}e^{-\lambda t}\, dt 
\label{eq:gammadistribution} 
\end{align}
$\lambda$ is not fundamental and may  be set to  $\lambda=1$ by  change of scale 
$t\to\lambda t$, while $\mu$ controls the shape of $G(t\vert \mu,\lambda)$.
The marginal~(\ref{eq:marginaldistribution})  becomes $M_{\alpha}(x\vert \mu,\lambda)$, with Laplace transform
\begin{align}
\int_0^\infty e^{-s x} M_\alpha(x\vert \mu,\lambda) \,dx &=  \frac{1}{s}\left(\frac{\lambda}{\lambda+ s^\alpha}\right)^\mu
       = \frac{1}{s}\left(1-\frac{s^\alpha}{\lambda+ s^\alpha}\right)^\mu
\label{eq:marginaldistributionLaplace}
\end{align}
We may now state Feller's approach  from a probabilistic  perspective.

\subsection{A Probabilistic View of Feller's Approach}
\label{sec:probFeller}

The case $\mu=1$ in~(\ref{eq:gammadistribution}) gives the exponential distribution $dG(t\vert \lambda) =  \lambda e^{-\lambda t}dt$.
Then $M_\alpha(x\vert \lambda)\equiv M_\alpha(x\vert \mu=1,\lambda)$ is 
\begin{align}
M_\alpha(x\vert \lambda) &= \int_0^\infty F_\alpha(x\vert t) dG(t\vert \lambda)
    = \lambda \int_0^\infty F_\alpha(x\vert t) e^{-\lambda t}\,dt
\label{eq:expprior} 
\end{align}
The Laplace transform of $M_\alpha(x\vert \lambda)$, read from~(\ref{eq:marginaldistributionLaplace}) with $\mu=1$, is
\begin{alignat}{3}
\int_0^\infty e^{-s x} &M_\alpha(x\vert \lambda) \,dx 
&&= \frac{1}{s}- \frac{s^{\alpha-1}}{\lambda+s^\alpha}
\label{eq:LaplaceMarginaldistribution} \\
\implies 
&\quad M_\alpha(x\vert \lambda) &&= 1-E_\alpha(-\lambda x^\alpha) \\
\implies 
&E_\alpha(-\lambda x^\alpha) &&= 1-M_\alpha(x\vert \lambda)
  = \lambda \int_0^\infty (1-F_\alpha(x\vert t)) e^{-\lambda t} \,dt
\label{eq:MLmarginal}
\end{alignat}
This reproduces Feller's result~(\ref{eq:FellerPollard}) 
from a probabilistic  perspective.
The difference is purely a matter  of conceptual outlook: 
\begin{description}
\item[Feller:] Study  the  two-dimensional Laplace transform of the bivariate distribution $1-F_\alpha(x t^{-1/\alpha})$, 
where $F_\alpha$ is the stable distribution. 
Deduce that $E_\alpha(-\lambda x^\alpha)/\lambda$ is the Laplace  transform of $1-F_\alpha(x t^{-1/\alpha})$ over $t$ at fixed $x$,
where $\lambda$ is the Laplace variable.
\item[Probabilistic:] Assign an exponential prior distribution $G(t\vert 1,\lambda)$ to the scale factor $t$ of the stable distribution 
$F_\alpha(x\vert t)\equiv F_\alpha(x t^{-1/\alpha})$, where $G(t\vert\mu,\lambda)$  is the gamma distribution.
Marginalise 
over $t$ to generate the Feller result directly.
\end{description}
Feller himself might also have established the result by the latter reasoning. 
Under subordination of processes, Feller~\cite{Feller2}(p451) discussed   mixture distributions but he did not  specifically discuss the Mittag-Leffler function in this context in his published work.
 The task fell on Pillai~\cite{Pillai} 
 to study $M_\alpha(x\vert\mu)\equiv M_\alpha(x\vert \mu, \lambda=1)$,
 including its infinite divisibility and the corresponding  Mittag-Leffler  stochastic process. 
 He also proved that  $M_\alpha(x\vert 1)=1-E_\alpha(-x^\alpha)$ (as  discussed  above), 
which he referred to as the Mittag-Leffler distribution.
There are thus two distributions bearing the  name ``Mittag-Leffler distribution'': 
$M_\alpha(x)=1-E_\alpha(-x^\alpha)$ and $P_\alpha(t) =1-  F_\alpha(t^{-1/\alpha})$.


The natural question arising from the probabilistic  approach is whether there might be other choices of $\mu$ 
in $G(\mu,\lambda)$ (or indeed other choices of $G$ altogether) 
that yield the Pollard result and, if so,  what   insight they might  offer.
At face value, there would appear to be nothing further to be said since other choices of $\mu$ 
can be expected to lead to different results,  beyond  the study of the Mittag-Leffler function. 
The main contribution of this paper is that, in fact,  there is a limit relationship that generates the Pollard result
for any $\mu>0$, as discussed next.

We first note, 
given the definition of the conditional stable density
\begin{align*}
f_\alpha(x\vert t)\equiv f_\alpha(x\,t^{-1/\alpha}) t^{-1/\alpha}
&\implies f_\alpha(1\vert t)\equiv f_\alpha(t^{-1/\alpha}) t^{-1/\alpha}
\end{align*}
that we may write $P_\alpha(t)$ of~(\ref{eq:PollardP1}) and the representation~(\ref{eq:LaplacePollardvariant}) 
of the Pollard result as 
\begin{align}
P_\alpha(t) &=  \int_0^t p_\alpha(u)\,du = \frac{1}{\alpha} \int_0^t f_\alpha(1\vert u) \, u^{-1}  \,du
\label{eq:Pollard1a} \\
\alpha E_\alpha(-\lambda x^\alpha) &= x \int_0^\infty f_\alpha(x\vert t) \, t^{-1} e^{-\lambda t} \,dt  \qquad 0<\alpha<1 
\label{eq:MLBayes1} \\
u=x^{-\alpha}t:\quad
E_\alpha(-\lambda x^\alpha) &= \int_0^\infty e^{-\lambda x^\alpha u} \, dP_\alpha(u)
\label{eq:MLBayes2}
\end{align} 
The intent  is to generate this representation using the general $G(\mu,\lambda)$ prior distribution,
{\it i.e.} without reference to  Pollard's analytic method and without explicit restriction to the $G(\mu=1,\lambda)$ case that 
is equivalent to Feller's approach, as demonstrated above.

\section{Main Contribution}
\label{sec:contribution}

We first  state Theorem~\ref{thm:main}, which warrants dedicated discussion,  even though it
is actually a special case of the  more general  Theorem~\ref{thm:ML3par} stated later.
We note first that the  density of the marginal distribution  $M_\alpha(x|\mu,\lambda)$ of Section~\ref{sec:probabilistic} is
\begin{align}
m_\alpha(x\vert \mu, \lambda) &= \int_0^\infty f_\alpha(x\vert t)\, dG(t\vert \mu, \lambda) \qquad \mu>0, \lambda>0 \nonumber \\
&= \frac{\lambda^\mu}{\Gamma(\mu)} \int_0^\infty f_\alpha(x\vert t)\, t^{\mu-1} e^{-\lambda t} \,dt  \nonumber \\
&= \frac{\mu \lambda^\mu}{\Gamma(\mu+1)} \int_0^\infty f_\alpha(x\vert t)\, t^{\mu-1} e^{-\lambda t} \,dt
\label{eq:marginaldensity}
\end{align} 
where the latter expression follows from the identity  $\mu\Gamma(\mu)=\Gamma(\mu+1)$.

\begin{theorem}
The limit
\begin{align}
\lim_{n\to\infty} \tfrac{n}{\mu} \, x \, m_\alpha(x\vert \tfrac{\mu}{n},\lambda) 
&= \lim_{n\to\infty} \tfrac{n}{\mu}  \, x \int_0^\infty f_\alpha(x\vert t)\, dG(t\vert \tfrac{\mu}{n},\lambda)  
\label{eq:limitmarginaldensity}
\intertext{is finite and independent of $\mu$ for any $\mu>0$.
This limit yields the following integral representation of the Mittag-Leffler function 
$E_\alpha(-\lambda x^\alpha)$}
\alpha E_\alpha(-\lambda x^\alpha)  &= x \int_0^\infty f_\alpha(x\vert t) \, t^{-1} e^{-\lambda t} \,dt
\label{eq:ML_intrep0} \\
u=x^{-\alpha}t:\quad
E_\alpha(-\lambda x^\alpha) 
&=  \int_0^\infty  e^{-\lambda x^\alpha u}\, dP_\alpha(u)
\label{eq:PollardML_intrep0}
 \intertext{where $P_\alpha(t)$ is  the (one-parameter) Pollard distribution}
P_\alpha(t) &= \frac{1}{\alpha} \int_0^t f_\alpha(1\vert u) \, u^{-1}  \,du \nonumber \\
 &= \frac{1}{\alpha} \int_0^t  f_\alpha(u^{-1/\alpha}) \, u^{-1/\alpha-1}  \,du \nonumber
\end{align} 
Hence $E_\alpha(-x)$  is completely monotone. 
\label{thm:main}
\end{theorem}

\begin{proof}[Proof of Theorem \ref{thm:main}]
The Laplace transform 
of $x \,m_\alpha(x\vert \mu,\lambda)$ is
\begin{align*}
\int_0^\infty e^{-sx} x \,m_\alpha(x\vert \mu,\lambda) \, dx
&= \int_0^\infty e^{-sx} x \int_0^\infty f_\alpha(x\vert t)\, dG(t\vert \mu, \lambda) \, dx \\
&= -\frac{d}{ds} \int_0^\infty \int_0^\infty  e^{-sx}  f_\alpha(x\vert t) \, dx \, dG(t\vert \mu, \lambda)  \\
&= -\frac{d}{ds} \int_0^\infty  e^{-t s^\alpha} \, dG(t\vert \mu, \lambda)   \\
&= \alpha s^{\alpha-1}  \int_0^\infty t \,e^{-t s^\alpha} \, dG(t\vert \mu, \lambda)   \\
&= \alpha s^{\alpha-1}  \frac{\lambda^\mu}{\Gamma(\mu)}  \int_0^\infty t^\mu \, e^{-(\lambda+s^\alpha) t}  \, dt  \\
&= \alpha s^{\alpha-1} \frac{\lambda^\mu}{\Gamma(\mu)}  \frac{\Gamma(\mu+1)}{(\lambda+s^\alpha)^{\mu+1}} \\
&=  \lambda^\mu \mu \alpha \frac{s^{\alpha-1} }{(\lambda+s^\alpha)^{\mu+1}} \\
\implies \;
\lim_{n\to\infty} \tfrac{n}{\mu} \int_0^\infty e^{-sx} x \, &m_\alpha(x\vert \tfrac{\mu}{n},\lambda) \, dx 
= \alpha \frac{s^{\alpha-1}}{\lambda+s^\alpha} 
\intertext{which is the Laplace transform of $\alpha E_\alpha(-\lambda x^\alpha)$.
With the aid of~(\ref{eq:marginaldensity}), it also readily follows that the limit~(\ref{eq:limitmarginaldensity}) is} 
\lim_{n\to\infty} \tfrac{n}{\mu} \, x \, m_\alpha(x\vert \tfrac{\mu}{n},\lambda)
&= x \int_0^\infty f_\alpha(x\vert t) \, t^{-1} e^{-\lambda t} \,dt 
\end{align*}
The integral representations (\ref{eq:ML_intrep0}) and (\ref{eq:PollardML_intrep0}) of $E_\alpha(-\lambda x^\alpha)$   follow,
hence the  conclusion that  $E_\alpha(-x)$ is completely monotone.
\end{proof}

Pursuing the probabilistic  theme, we  turn next to Laplace convolution  
to demonstrate the complete monotonicity of  the two and three parameter Mittag-Leffler functions.

\section{A Convolution  Representation}
\label{sec:convolution}

Toward a more general discussion, we first  present an alternative representation of $x f_\alpha(x|t)$ using Laplace convolution.
 The convolution $\{\rho \star f\}(x)$ of  $\rho(x), f(x)$ 
 is given by
  \begin{align}
 \{\rho\star f\}(x) &= \int_0^x \rho(x-u)f(u)\, du  
 \label{eq:convolution}
\end{align}
The  convolution theorem states that the Laplace transform of  $\{\rho \star f\}$ is a product of the 
Laplace transforms of  $\rho, f$.

\subsection{One Parameter Case}
\label{sec:1par}

\begin{proposition}
Let $\rho_\alpha(x) = x^{-\alpha}/\Gamma(1-\alpha)$, $0<\alpha<1$  with Laplace transform 
$s^{\alpha-1}$.
Let  $\{\rho_\alpha\star f_\alpha(\cdot\vert t)\}(x)$ be  the convolution of $\rho_\alpha(x)$ and $f_\alpha(x\vert t)$
with Laplace transform $s^{\alpha-1}e^{-ts^\alpha}$.
Then  
\begin{align}
x\, f_\alpha(x\vert t) =\alpha\, t\{\rho_\alpha\star f_\alpha(\cdot\vert t)\}(x) = \alpha\, \{\rho_\alpha\star f_\alpha\}(xt^{-1/\alpha})
\label{eq:convolution1}
\end{align}
where $\{\rho_\alpha\star f_\alpha\}(x)$ is the convolution of $\rho_\alpha(x)$ and $f_\alpha(x)\equiv f_\alpha(x|1)$.
For compatibility with later discussion, we also use  the name $w_\alpha(x\vert t)$ defined by
$\alpha w_\alpha(x\vert t) \equiv x\, f_\alpha(x\vert t)$.
\label{prop:convolution1}
\end{proposition}

\begin{proof}[Proof of Proposition  \ref{prop:convolution1}]
By the  convolution theorem, $\{\rho_\alpha\star f_\alpha(\cdot\vert t)\}(x)$ has Laplace transform 
\begin{align*} 
s^{\alpha-1} e^{-ts^\alpha} = -\frac{1}{\alpha t}\frac{d}{ds} e^{-ts^\alpha} 
 &=  \frac{1}{\alpha t} \int_0^\infty e^{-sx} x f_\alpha(x\vert t)\, dx \nonumber \\
\implies\; 
 \alpha\, t\,  \{\rho_\alpha\star f_\alpha(\cdot\vert t)\}(x) &= x f_\alpha(x\vert t)   
\end{align*}
The convolution $\{\rho_\alpha\star f_\alpha(\cdot\vert t)\}(x)$ takes the explicit form:
\begin{align*} 
\{\rho_\alpha\star f_\alpha(\cdot\vert t)\}(x) &= \int_0^x \rho_\alpha(x-u)f_\alpha(u\vert t)\, du  \\
&= \int_0^x \rho_\alpha(x-u)f_\alpha(u t^{-1/\alpha})t^{-1/\alpha} \, du  \\
y=ut^{-1/\alpha}:\quad
&= \int_0^{xt^{-1/\alpha}} \rho_\alpha(x-yt^{1/\alpha})f_\alpha(y) \, dy  \\
&= \int_0^{xt^{-1/\alpha}} \rho_\alpha(t^{1/\alpha}(xt^{-1/\alpha}-y))f_\alpha(y) \, dy  \\
&= t^{-1}\int_0^{xt^{-1/\alpha}} \rho_\alpha(xt^{-1/\alpha}-y)f_\alpha(y) \, dy  \\
&= t^{-1} \{\rho_\alpha\star f_\alpha\}(xt^{-1/\alpha}) 
\end{align*}
\label{proof:prop:convolution1}
so that $\alpha w_\alpha(x\vert t) \equiv
x\, f_\alpha(x\vert t) =\alpha\, t\{\rho_\alpha\star f_\alpha(\cdot\vert t)\}(x) = \alpha\, \{\rho_\alpha\star f_\alpha\}(xt^{-1/\alpha})$.
\end{proof}

Hence the following are equivalent representations of  the  Pollard distribution $P_\alpha(t)$:
\begin{align} 
P_\alpha(t) &= \int_0^t w_\alpha(1\vert t) \, u^{-1}  \,du
 \equiv  \frac{1}{\alpha} \int_0^t f_\alpha(1\vert u) \, u^{-1}  \,du  \nonumber \\
 &= \int_0^t \{\rho_\alpha\star f_\alpha(\cdot\vert u)\}(1)   \, du  \nonumber \\
 &= \int_0^t \{\rho_\alpha\star f_\alpha\}(u^{-1/\alpha})\,  u^{-1}  \, du 
 \label{eq:PollardP1par} 
 \end{align}

The motivation for the convolution representation  is  to facilitate generalisation.
Specifically, the Laplace transform $\alpha t s^{\alpha-1}e^{-ts^\alpha}$ of $xf_\alpha(x\vert t)$ 
is    the derivative of $-e^{-ts^\alpha}$.
However, a more general term like $t s^{\alpha-\beta}e^{-ts^\alpha}$  cannot arise from simple derivatives of 
$e^{-ts^\alpha}$ for non-integer $\beta$.
It might  be interpreted as a fractional derivative, as can be represented instead by  convolutions.
Accordingly, we  proceed to consider more general convolutions than the  
convolution form~(\ref{eq:convolution1}) for $xf_\alpha(x\vert t)$.

\subsection{Two Parameter Case}
\label{sec:2par}

First, we introduce the two-parameter Mittag-Leffler function  
\begin{align}
E_{\alpha,\beta}(x) &= \sum_{k=0}^\infty \frac{x^k}{\Gamma(\alpha k+\beta)} 
\label{eq:ML2parseries}
\end{align}
The Laplace transform of $x^{\beta-1}E_{\alpha,\beta}(-\lambda x^\alpha)$  is
\begin{align}
\int_0^\infty e^{-sx} x^{\beta-1}E_{\alpha,\beta}(-\lambda x^\alpha) \,dx  &= \frac{s^{\alpha-\beta}}{\lambda+s^\alpha} 
\label{eq:LaplaceML2par}
\end{align}
We may now proceed to prove  that $E_{\alpha,\beta}(-x)$ is completely monotone
by showing  that it is the Laplace  transform of a two-parameter variant $P_{\alpha,\beta}(t)$ of the Pollard distribution.
We follow a corresponding  two-parameter variant of the convolution argument presented above for the one-parameter case.
\begin{proposition}
Let $\rho_{\alpha,\beta}(x) = x^{\beta-\alpha-1}/\Gamma(\beta-\alpha)$ $\beta>\alpha$, with Laplace transform 
$s^{\alpha-\beta}$.
Let  $\{\rho_{\alpha,\beta}\star f_\alpha(\cdot\vert t)\}(x)$  
be the convolution of $\rho_{\alpha,\beta}(x)$ and $f_\alpha(x\vert t)$. 
Then  
\begin{align}
w_{\alpha,\beta}(x\vert t) \equiv
t\, \{\rho_{\alpha,\beta}\star f_\alpha(\cdot\vert t)\}(x) &=  t^{(\beta-1)/\alpha}\ \{\rho_{\alpha,\beta}\star f_\alpha\}(xt^{-1/\alpha})
\label{eq:convolution2}
\end{align}
$($the name $w_{\alpha,\beta}(x\vert t)$ is a shorthand adopted for convenience$)$.
\label{prop:convolution2}
\end{proposition}
\begin{proof}[Proof of Proposition  \ref{prop:convolution2}]
\begin{align*} 
\{\rho_{\alpha,\beta}\star f_\alpha(\cdot\vert t)\}(x) &= \int_0^x \rho_{\alpha,\beta}(x-u)f_\alpha(u\vert t)\, du  \\
&= \int_0^{xt^{-1/\alpha}} \rho_{\alpha,\beta}(t^{1/\alpha}(xt^{-1/\alpha}-u))f_\alpha(u) \, du  \\
&= t^{(\beta-1)/\alpha-1} \int_0^{xt^{-1/\alpha}} \rho_{\alpha,\beta}(xt^{-1/\alpha}-u)f_\alpha(u) \, du  \\
&=  t^{(\beta-1)/\alpha-1} \{\rho_{\alpha,\beta}\star f_\alpha\}(xt^{-1/\alpha})  \\
\end{align*}
\label{proof:prop:convolution2}
Thus $w_{\alpha,\beta}(x\vert t) \equiv t\,\{\rho_{\alpha,\beta}\star f_\alpha(\cdot\vert t)\}(x)
= t^{(\beta-1)/\alpha} \{\rho_{\alpha,\beta}\star f_\alpha\}(xt^{-1/\alpha})$.
\end{proof}

\begin{theorem}
The   two-parameter Mittag-Leffler function $E_{\alpha,\beta}(-\lambda x^\alpha)$ has the integral representation
\begin{align} 
E_{\alpha,\beta}(-\lambda x^\alpha) &= \int_0^\infty e^{-\lambda x^\alpha t} \, dP_{\alpha,\beta}(t)
 \label{eq:ML2par} 
 \intertext{where $P_{\alpha,\beta}(t)$, which we refer to as  the two-parameter Pollard distribution, is}
 P_{\alpha,\beta}(t) &= \int_0^t w_{\alpha,\beta}(1\vert u) \, u^{-1}\, du \nonumber \\
 &\equiv \int_0^t \{\rho_{\alpha,\beta}\star f_\alpha(\cdot\vert u)\}(1)  \, du \nonumber \\
 &= \int_0^t \{\rho_{\alpha,\beta}\star f_\alpha\}(u^{-1/\alpha})\,  u^{(\beta-1)/\alpha-1}  \, du 
 \label{eq:PollardP2par}
 \end{align}
Hence $E_{\alpha,\beta}(-x)$ is completely monotone.
\label{thm:ML2par}
\end{theorem}
\begin{proof}[Proof of Theorem \ref{thm:ML2par}]
The theorem is a particular case of the more general Theorem~\ref{thm:ML3par} below, hence the current proof is deferred 
to that of the latter theorem.
\end{proof}

\subsection{Three Parameter Case}
\label{sec:3par}

The three-parameter Mittag-Leffler function, also known as the Prabhakar function,  is given by  
\begin{align}
E^\gamma_{\alpha,\beta}(x) &= \frac{1}{\Gamma(\gamma)} 
   \sum_{k=0}^\infty \frac{\Gamma(\gamma+k)}{k!\,\Gamma(\alpha k+\beta)}\, x^k  
\label{eq:ML3parseries}
\end{align}
The Laplace transform of $x^{\beta-1}E^\gamma_{\alpha,\beta}(-\lambda x^\alpha)$  is
\begin{align}
\int_0^\infty e^{-sx} x^{\beta-1}E^\gamma_{\alpha,\beta}(-\lambda x^\alpha) \,dx  
  &= \frac{s^{\alpha\gamma-\beta}}{(\lambda+s^\alpha)^\gamma} 
\label{eq:LaplaceML3par}
\end{align}
We may now proceed to prove  that $E^\gamma_{\alpha,\beta}(-x)$ is completely monotone
by showing  that it is the Laplace  transform of a three-parameter variant $P^\gamma_{\alpha,\beta}(t)$ of the Pollard distribution.
In principle, we need only have discussed  the three-parameter case from the outset because 
the two and one-parameter instances are the special cases $\gamma=1$ and  $\gamma=\beta=1$ respectively.
We chose instead to present in  sequential order for   clarity of exposition.

We devote a separate section to the  three-parameter case, which subsumes all prior discussion, 
by restating Theorem~\ref{thm:main} in the three-parameter context.

\section{Main  Theorem}
\label{sec:genmain}

We start with a  proposition required for the general theorem that follows:
\begin{proposition}
Let $\rho^\gamma_{\alpha,\beta}(x) = x^{\beta-\alpha\gamma-1}/\Gamma(\beta-\alpha\gamma)$ 
$(0<\alpha<1, \gamma>0, \beta>\alpha\gamma)$ 
and let  $\{\rho^\gamma_{\alpha,\beta}\star f_\alpha(\cdot\vert t)\}(x)$  
be the convolution of $\rho^\gamma_{\alpha,\beta}(x)$ and the stable density $f_\alpha(x\vert t)$.
Then  
\begin{align}
w^\gamma_{\alpha,\beta}(x\vert t) \equiv
t^\gamma \, \{\rho^\gamma_{\alpha,\beta}\star f_\alpha(\cdot\vert t)\}(x) 
&=  t^{(\beta-1)/\alpha} \{\rho^\gamma_{\alpha,\beta}\star f_\alpha\}(xt^{-1/\alpha})
\label{eq:convolution3}
\end{align}
\label{prop:convolution3}
\end{proposition}
\begin{proof}[Proof of Proposition  \ref{prop:convolution3}]
\begin{align*} 
\{\rho^\gamma_{\alpha,\beta}\star f_\alpha(\cdot\vert t)\}(x) &= \int_0^x \rho^\gamma_{\alpha,\beta}(x-u)f_\alpha(u\vert t)\, du  \\
&= \int_0^{xt^{-1/\alpha}} \rho^\gamma_{\alpha,\beta}(t^{1/\alpha}(xt^{-1/\alpha}-u))f_\alpha(u) \, du  \\
&= t^{(\beta-1)/\alpha-\gamma} \int_0^{xt^{-1/\alpha}} \rho^\gamma_{\alpha,\beta}(xt^{-1/\alpha}-u)f_\alpha(u) \, du  \\
&=  t^{(\beta-1)/\alpha-\gamma} \{\rho^\gamma_{\alpha,\beta}\star f_\alpha\}(xt^{-1/\alpha})  
\end{align*}
Thus $w^\gamma_{\alpha,\beta}(x\vert t) \equiv t^\gamma \,\{\rho^\gamma_{\alpha,\beta}\star f_\alpha(\cdot\vert t)\}(x)
= t^{(\beta-1)/\alpha} \{\rho^\gamma_{\alpha,\beta}\star f_\alpha\}(xt^{-1/\alpha})$.
\label{proof:convolution3}
\end{proof}

\begin{theorem}
Let $\rho^\gamma_{\alpha,\beta}(x), w^\gamma_{\alpha,\beta}(x\vert t)$ $(0<\alpha<1, \gamma>0, \beta>\alpha\gamma)$
be as defined in Proposition~\ref{prop:convolution3}
and let $G(\mu, \lambda)$ be the gamma distribution with shape and scale parameters $\mu>0, \lambda>0$ respectively.
Let the  distribution $M^\gamma_{\alpha,\beta}(x|\mu,\lambda)$ 
have  density 
\begin{align}
m^\gamma_{\alpha,\beta}(x|\mu,\lambda)
&= \int_0^\infty w^\gamma_{\alpha,\beta}(x\vert t) \, dG(t\vert \mu, \lambda) \nonumber \\
&= \frac{\lambda^\mu}{\Gamma(\mu)} \int_0^\infty w^\gamma_{\alpha,\beta}(x\vert t)  \, t^{\mu-1} e^{-\lambda t} \,dt 
\label{eq:genmarginaldensity} \\
&\equiv \frac{\lambda^\mu}{\Gamma(\mu)} \int_0^\infty \{\rho^\gamma_{\alpha,\beta}\star f_\alpha(\cdot\vert t)\}(x) 
\, t^{\gamma+\mu-1} e^{-\lambda t} \,dt  
\label{eq:genmarginaldensity1} \\
&= \frac{\lambda^\mu}{\Gamma(\mu)} \int_0^\infty  \{\rho^\gamma_{\alpha,\beta}\star f_\alpha\}(xt^{-1/\alpha})
\, t^{(\beta-1)/\alpha+\mu-1} e^{-\lambda t} \,dt  
\label{eq:genmarginaldensity2} 
\end{align} 
where the latter two forms follow from Proposition~\ref{prop:convolution3}.
Then the  following limit is finite and independent of $\mu$ for any $\mu>0$
\begin{align}
\lim_{n\to\infty} \tfrac{n}{\mu} \, m^\gamma_{\alpha,\beta}(x\vert \tfrac{\mu}{n},\lambda) 
\label{eq:limitgenmarginaldensity}
\end{align}
This limit yields the following integral representation of the three-parameter Mittag-Leffler or Prabhakar  function 
$E^\gamma_{\alpha,\beta}(-\lambda x^\alpha)$
\begin{align}
E^\gamma_{\alpha,\beta}(-\lambda x^\alpha)  &= \int_0^\infty w^\gamma_{\alpha,\beta}(x\vert t)  \, t^{-1} e^{-\lambda t} \,dt
=  \int_0^\infty  e^{-\lambda x^\alpha t}\, dP^\gamma_{\alpha,\beta}(t)
\label{eq:ML3par} 
 \intertext{where $P^\gamma_{\alpha,\beta}(t)$, which we refer to as  the three-parameter Pollard distribution, is}
 P^\gamma_{\alpha,\beta}(t) &= \int_0^t w^\gamma_{\alpha,\beta}(1\vert u) \, u^{-1}\, du \nonumber \\
 &\equiv  \frac{1}{\Gamma(\gamma)} \int_0^t \{\rho^\gamma_{\alpha,\beta}\star f_\alpha(\cdot\vert u)\}(1)\, u^{\gamma-1}  \, du \nonumber \\
 &= \frac{1}{\Gamma(\gamma)} \int_0^t \{\rho^\gamma_{\alpha,\beta}\star f_\alpha\}(u^{-1/\alpha})\,  u^{(\beta-1)/\alpha-1}  \, du 
 \label{eq:PollardP3par}
 \end{align}
Hence $E^\gamma_{\alpha,\beta}(-x)$ is completely monotone.
\label{thm:ML3par}
\end{theorem}

\begin{proof}[Proof of Theorem \ref{thm:ML3par}]
The Laplace transform $\widetilde m^\gamma_{\alpha,\beta}(s \vert \mu,\lambda)$ 
of (\ref{eq:genmarginaldensity})  is
\begin{align}
\widetilde m^\gamma_{\alpha,\beta}(s \vert \mu,\lambda)
&\equiv \int_0^\infty e^{-sx} \, m^\gamma_{\alpha,\beta}(x\vert \mu,\lambda)\, dx \nonumber \\
&= s^{\alpha\gamma-\beta}  \frac{\lambda^\mu}{\Gamma(\mu)} \int_0^\infty 
\, t^{\gamma+\mu-1} e^{-(\lambda+s^\alpha) t} \,dt  \nonumber \\
&=  \lambda^\mu \frac{\Gamma(\gamma+\mu) }{\Gamma(\mu)} 
\frac{s^{\alpha\gamma-\beta}}{(\lambda+s^\alpha)^{\gamma+\mu}} 
\label{eq:genmarginaldensityLaplace} \\
\implies \;
\lim_{n\to\infty} \tfrac{n}{\mu} \int_0^\infty  &e^{-sx} m^\gamma_{\alpha,\beta}(x\vert \tfrac{\mu}{n},\lambda) \, dx 
 = \Gamma(\gamma) 
\frac{s^{\alpha\gamma-\beta}}{(\lambda+s^\alpha)^{\gamma}} 
\label{eq:limitgenmarginaldensityLaplace}
\end{align}
By (\ref{eq:LaplaceML3par}),   the right hand side is the Laplace  transform of 
$\Gamma(\gamma)\,x^{\beta-1}E^\gamma_{\alpha,\beta}(-\lambda x^\alpha)$.
Given (\ref{eq:genmarginaldensity1}) and 
(\ref{eq:genmarginaldensity2}), it also readily follows that 
the limit~(\ref{eq:limitgenmarginaldensity}) is 
\begin{align} 
 \int_0^\infty t^\gamma \{\rho^\gamma_{\alpha,\beta}\star f_\alpha(\cdot\vert t)\}(x)
 &t^{-1} e^{-\lambda t} dt
 = \int_0^\infty  \{\rho^\gamma_{\alpha,\beta}\star f_\alpha\}(xt^{-1/\alpha}) \, t^{(\beta-1)/\alpha-1} e^{-\lambda t} dt \nonumber \\
\implies\; E^\gamma_{\alpha,\beta}(-\lambda x^\alpha)
 &= \frac{x^{1-\beta}}{\Gamma(\gamma)}
  \int_0^\infty  \{\rho^\gamma_{\alpha,\beta}\star f_\alpha\}(xt^{-1/\alpha})\, t^{(\beta-1)/\alpha-1} e^{-\lambda t}  \,dt \nonumber \\
u=x^{-\alpha}t:\;
 &=  \frac{1}{\Gamma(\gamma)}
 \int_0^\infty  e^{-\lambda x^\alpha u}\, \{\rho^\gamma_{\alpha,\beta}\star f_\alpha\}(u^{-1/\alpha})\, u^{(\beta-1)/\alpha-1}  \,du \nonumber \\
 &= \int_0^\infty e^{-\lambda x^\alpha u} \, dP^\gamma_{\alpha,\beta}(u) \nonumber 
\end{align}
Hence $E^\gamma_{\alpha,\beta}(-x)$ is completely monotone.
\label{proof:ML3par}
\end{proof}
Theorem~\ref{thm:ML3par} may  be  visually represented by the following commutative diagram,
 where $m_{\alpha,\beta}^\gamma(x\vert\mu,\lambda)$ and  its Laplace transform
$\widetilde{m}_{\alpha,\beta}^\gamma(s\vert\mu,\lambda)$
are given by (\ref{eq:genmarginaldensity}) and (\ref{eq:genmarginaldensityLaplace}) respectively.
The equivalence of the two routes from the top left node to the bottom left node induces the integral representation of the Mittag-Leffler function.
\begin{equation}
\begin{tikzpicture}[auto,scale=3, baseline=(current  bounding  box.center)]
\newcommand*{\size}{\small}%
\newcommand*{\gap}{.2ex}%
\newcommand*{\width}{1.7}%
\newcommand*{\height}{1.25}%

\node (P) at (0,0)  [align=center]
{$m_{\alpha,\beta}^\gamma(x\vert\mu,\lambda)$};
\node (Q) at ($(P)+(\width,0)$) [align=center]
{$\widetilde{m}_{\alpha,\beta}^\gamma(s\vert\mu,\lambda)$};
\node (B) at ($(P)-(0,\height)$) {$\Gamma(\gamma)x^{\beta-1}E_{\alpha,\beta}^\gamma(-\lambda x^\alpha)$}; 
\node (C) at ($(B)+(\width,0)$) {$\Gamma(\gamma)\dfrac{s^{\alpha\gamma-\beta}}{(\lambda+s^\alpha)^\gamma}$};   
\draw[Myarrow] ([yshift =  \gap]P.east)  --  node[above] {\size $\scrL$} ([yshift = \gap]Q.west) ;
\draw[Myarrow]([xshift  =  \gap]Q.south) --  node[right] {\size $\lim\limits_{n\to \infty}\tfrac{n}{\mu}\;
\widetilde{m}_{\alpha,\beta}^\gamma(s\vert \tfrac{\mu}{n},\lambda)$} ([xshift =  \gap]C.north);
\draw[Myarrow] ([xshift =  \gap]P.south) -- node[left]  {\size $\lim\limits_{n\to \infty}\tfrac{n}{\mu}\;
m_{\alpha,\beta}^\gamma(x\vert \tfrac{\mu}{n},\lambda)$} ([xshift =  \gap]B.north);
\draw[Myarrow] ([yshift = +\gap]C.west) --  node[below] {\size $\scrL^{-1}$} ([yshift  = +\gap]B.east); 
\end{tikzpicture}
\label{eq:CommutativeDiagram}
\end{equation}

The  representation~(\ref{eq:ML3par}) of $E^\gamma_{\alpha,\beta}(x)$, 
with $P^\gamma_{\alpha,\beta}(t)$ given by (\ref{eq:PollardP3par}),  
is equivalent to equation (2.4) in G\'{o}rska et al.~\cite{Gorska}. 
The  difference is one of approach.
This paper offers a   fundamentally  probabilistic argument,
while G\'{o}rska et al.~\cite{Gorska} follows a complex analytic route inspired by  Pollard~\cite{PollardML}.
The balance of G\'{o}rska et al.~\cite{Gorska}  is devoted to  finding an explicit formula for a function $f^\gamma_{\alpha,\beta}(x)$ featuring in
the paper in terms of the Meijer $G$ function and associated confluent Wright function.
In turns out that  $f^\gamma_{\alpha,\beta}(x)$  in G\'{o}rska et al.~\cite{Gorska} is identical to  
$\{\rho^\gamma_{\alpha,\beta}\star f_\alpha\}(x)$ in this paper.
We are content to leave it in the conceptually simple convolution form:
\begin{align} 
\{\rho^\gamma_{\alpha,\beta}\star f_\alpha\}(x) &= \int_0^x \rho^\gamma_{\alpha,\beta}(x-u) f_\alpha(u)\, du \nonumber  \\
   &= \frac{1}{\Gamma(\beta-\alpha\gamma)}  \int_0^x  (x-u)^{\beta-\alpha\gamma-1}  f_\alpha(u)\, du 
 \label{eq:rhostarstablef}
\end{align}
rather than express it in terms of   special functions.
In our context, we have actually worked with the conditional   density 
\begin{align*}
w^\gamma_{\alpha,\beta}(x\vert t) \equiv
t^\gamma \, \{\rho^\gamma_{\alpha,\beta}\star f_\alpha(\cdot\vert t)\}(x) 
&=  t^{(\beta-1)/\alpha} \{\rho^\gamma_{\alpha,\beta}\star f_\alpha\}(xt^{-1/\alpha})
\end{align*}
where we assigned a gamma prior distribution  to the scale parameter $t$.
The density $w^\gamma_{\alpha,\beta}(x\vert t) $  reduces to (\ref{eq:rhostarstablef}) for the particular choice $t=1$.

We have  completed the  task of proving that  the three-parameter Mittag-Leffler function  
$E^\gamma_{\alpha,\beta}(-x)$ is completely monotone by methods of probability theory,
using probabilistic reasoning  to derive an explicit  form for $P^\gamma_{\alpha,\beta}(t)$,
whose Laplace transform is $E^\gamma_{\alpha,\beta}(-x)$.
Beyond that, 
we draw  conclusions  on the complete monotonicity of related functions, notably
$x^{\beta-1} E^\gamma_{\alpha,\beta}(-x^\alpha)$ and $ E^\gamma_{\alpha,\beta}(-x^\alpha)$ in isolation.
First, we discuss $x^{\beta-1} E^\gamma_{\alpha,\beta}(-x^\alpha)$, the bottom left node of the 
commutative diagram~(\ref{eq:CommutativeDiagram}),  
 in the probabilistic  context of Theorem~\ref{thm:ML3par}.
 The discussion involves an alternative representation of the fundamental probabilistic  object --  the convolution  
density  $\{\rho^\gamma_{\alpha,\beta}\star f_\alpha(\cdot\vert t)\}(x)$. 

\section{An Alternative Representation}
\label{sec:alternative}

For $x^{\beta-1}E^\gamma_{\alpha,\beta}(-\lambda x^\alpha)$  to be   completely monotone, there must exist a  
distribution $R^\gamma_{\alpha,\beta}(u\vert\lambda)$ defined  by the Laplace transform
\begin{align}
x^{\beta-1}E^\gamma_{\alpha,\beta}(-\lambda x^\alpha) &= \int_0^\infty e^{- x u} \, dR^\gamma_{\alpha,\beta}(u\vert\lambda) 
\label{eq:LaplaceRabg}
\intertext{In turn, the Laplace transform of~(\ref{eq:LaplaceRabg})  is the 
Stieltjes transform (or iterated Laplace transform) of $R^\gamma_{\alpha,\beta}(u\vert\lambda)$:}
\frac{s^{\alpha\gamma-\beta}}{(\lambda+s^\alpha)^\gamma}
&= \int_0^\infty \frac{1}{s+u}  \, dR^\gamma_{\alpha,\beta}(u\vert\lambda)
\label{eq:StieltjesRabg}
\end{align}

Then, as de Oliviera et al.~\cite{Oliveira}, Mainardi and Garrappa~\cite{MainardiGarrappa} show, 
the Stieltjes inversion formula (Titchmarsh~\cite{Titchmarsh}(11.8, p318), 
Widder~\cite{Widder}(VIII.7, p342))  gives 
\begin{align}
dR^\gamma_{\alpha,\beta}(u\vert\lambda)
&=  \frac{1}{\pi}\,{\rm Im}\left\{\frac{(e^{-i\pi}u)^{\alpha\gamma-\beta}}{(\lambda+(e^{-i\pi}u)^\alpha)^\gamma}\right\} \, du
\label{eq:Rabg} 
\intertext{
The expression in braces on the RHS of (\ref{eq:Rabg}) is (\ref{eq:StieltjesRabg}) at $s=e^{-i\pi}u$.
In particular, for $\gamma=\beta=1$, (\ref{eq:Rabg}) reduces to} 
dR_\alpha(u\vert\lambda)
&=    \frac{1}{\pi}\, \frac{\lambda\,u^{\alpha-1}\sin\pi\alpha}{\lambda^2+2\lambda\,u^\alpha\cos\pi\alpha+u^{2\alpha}} \, du
\label{eq:Ra}
\end{align}
which has been discussed  in various contexts in the fractional calculus  and probabilistic  literature
({\it e.g.}\ James~\cite{James_Lamperti} in the latter context).

We have mentioned (\ref{eq:LaplaceRabg})
for completeness but it was not the core of our probabilistic  discussion, whose focus  was to determine 
 $P^\gamma_{\alpha,\beta}(t)$, with Laplace transform $E^\gamma_{\alpha,\beta}(-x)$.
 That said, we can offer a `hybrid' derivation of (\ref{eq:LaplaceRabg}) that combines  the core of the probabilistic argument 
 in the form of  the convolution  density $\{\rho^\gamma_{\alpha,\beta}\star f_\alpha(\cdot\vert t)\}(x)$  
 with the complex analytic Stieltjes inversion argument presented above.

Assume  $\{\rho^\gamma_{\alpha,\beta}\star f_\alpha(\cdot\vert t)\}(x)$ to be the Laplace transform of a distribution  
$S^\gamma_{\alpha,\beta}(u\vert t)$:
\begin{align}
\{\rho^\gamma_{\alpha,\beta}\star f_\alpha(\cdot\vert t)\}(x) 
&= \int_0^\infty e^{- x u} \, dS^\gamma_{\alpha,\beta}(u\vert t) 
\label{eq:LaplaceSabg}
\intertext{In turn, the Laplace transform of~(\ref{eq:LaplaceSabg})  is the Stieltjes transform of $S^\gamma_{\alpha,\beta}(u\vert t)$:}
s^{\alpha\gamma-\beta} e^{-ts^\alpha}
&= \int_0^\infty \frac{1}{s+u}  \, dS^\gamma_{\alpha,\beta}(u\vert t)
\label{eq:StieltjesSabg}
\intertext{By the Stieltjes inversion formula:} 
dS^\gamma_{\alpha,\beta}(u\vert t)
&= \frac{1}{\pi} \, {\rm Im} \left\{(ue^{-i\pi})^{\alpha\gamma-\beta} e^{-t(ue^{-i\pi})^\alpha} \right\}\, du
\end{align}
Hence, using the representation (\ref{eq:LaplaceSabg}) in the proof of Theorem~\ref{thm:ML3par}:
\begin{align}
\Gamma(\gamma)\,x^{\beta-1}E^\gamma_{\alpha,\beta}(-\lambda x^\alpha)
&=  \int_0^\infty t^\gamma \{\rho^\gamma_{\alpha,\beta}\star f_\alpha(\cdot\vert t)\}(x) \, t^{-1} e^{-\lambda t} \,dt \nonumber \\
&=  \int_0^\infty dt \,  t^{\gamma-1} e^{-\lambda t} \int_0^\infty e^{-xu}  \, dS^\gamma_{\alpha,\beta}(u\vert t)  \nonumber  \\
&=   \frac{1}{\pi}\, {\rm Im} \int_0^\infty du \, e^{-xu} (ue^{-i\pi})^{\alpha\gamma-\beta}  
\int_0^\infty t^{\gamma-1} e^{-(\lambda+(ue^{-i\pi})^\alpha)t} \, dt  \nonumber  \\
&=   \frac{\Gamma(\gamma)}{\pi}\, {\rm Im} \int_0^\infty  \, e^{-xu} 
           \frac{(e^{-i\pi}u)^{\alpha\gamma-\beta}}{(\lambda+(e^{-i\pi}u)^\alpha )^\gamma} du  \nonumber  \\
&=   \Gamma(\gamma) \, \int_0^\infty e^{- x u} \, dR^\gamma_{\alpha,\beta}(u\vert\lambda)
\end{align}
thereby reproducing (\ref{eq:LaplaceRabg}).

The Stieltjes transform  and its complex analytic inverse are not unfamiliar  in   probability  theory.
In his study of a family of distributions known as generalised gamma convolutions, 
Bondesson~\cite{Bondesson}  used the concept  under the guise of  Pick functions (also known as Nevanlinna functions).

We turn next to the complete monotonicity of $E^\gamma_{\alpha,\beta}(-\lambda x^\alpha)$.

\section{A Further Consequence}
\label{sec:consequence}

There is a well-known  property  of completely monotone functions ({\it e.g.}\ Schilling et al.~\cite{Schilling})
that we state without proof in Proposition~\ref{prop:composition}.
We start with a definition:

\begin{definition}
A Bernstein function is a nonnegative function $\eta(x)$, $x\ge0$ with a completely monotone derivative, 
{\it i.e.} $\eta(x)\ge0$ and $(-1)^{k-1}\eta^{(k)}(x)\ge0$, $k\ge1$.  
For example, $\eta(x\vert\lambda)=\lambda x^\alpha$ ($0\le\alpha\le1, \lambda>0$) is a Bernstein function.
\label{def:Bernstein}
\end{definition}
\begin{proposition}
If $\varphi(x)$ is completely monotone and $\eta$ is a Bernstein function,  $\varphi(\eta)$ is completely monotone.

\label{prop:composition}
\end{proposition}
\begin{theorem}
Given a Bernstein function  $\eta$,
the  Mittag-Leffler function $E^\gamma_{\alpha,\beta}(-\eta)$  is completely monotone.
For example, 
$E^\gamma_{\alpha,\beta}(-\lambda x^\alpha)$ is completely monotone. 
\label{thm:composition}
\end{theorem}

\begin{proof}[Proof of Theorem \ref{thm:composition}]
We have already shown that $E^\gamma_{\alpha,\beta}(-x)$  is completely monotone.
Hence, by Proposition~\ref{prop:composition}, 
$E^\gamma_{\alpha,\beta}(-\eta)$ is completely monotone for  a Bernstein function $\eta$.
Specifically, $\eta(x\vert\lambda)=\lambda x^\alpha$ $(0\le\alpha\le1, \lambda>0)$ is a Bernstein function,
hence $E^\gamma_{\alpha,\beta}(-\lambda x^\alpha)$ is completely monotone. 
\label{proof:composition}
\end{proof}
The complete monotonicity of $E^\gamma_{\alpha,\beta}(-\lambda x^\alpha)$ implies that there exists a distribution  
$Q^\gamma_{\alpha,\beta}(t\vert\lambda)$ whose Laplace  transform  is $E^\gamma_{\alpha,\beta}(-\lambda x^\alpha)$:

\begin{align}
E^\gamma_{\alpha,\beta}(-\lambda x^\alpha) &= \int_0^\infty e^{- x t} \, dQ^\gamma_{\alpha,\beta}(t\vert\lambda) 
\label{eq:Qabg}
\end{align}
$Q^\gamma_{\alpha,\beta}(t\vert\lambda)$ is to $E^\gamma_{\alpha,\beta}(-\lambda x^\alpha)$ what 
$P^\gamma_{\alpha,\beta}(t)$  is to $E^\gamma_{\alpha,\beta}(-x)$.
However, determining $Q^\gamma_{\alpha,\beta}(t\vert\lambda)$  appears to be a  challenging problem,
whether the approach is analytic or probabilistic.

Clearly, (\ref{eq:Qabg}) and (\ref{eq:Rabg}) are identical for $\beta=1$, 
{\it i.e.}\ $Q^\gamma_{\alpha,1}(t\vert\lambda)\equiv R^\gamma_{\alpha,1}(t\vert\lambda)$.
But, to our awareness, determining  $Q^\gamma_{\alpha,\beta}(t\vert\lambda)$ for $\beta\ne1$ is an open problem.
We shall not pursue it further here.
 Our primary purpose in this section  was to bring attention to Theorem~\ref{thm:composition}
and hence the existence of a distribution $Q^\gamma_{\alpha,\beta}(t\vert\lambda)$ defined by~(\ref{eq:Qabg}).

\section{A Different  Generalisation}
\label{sec:tilting}

As mentioned in Section~\ref{sec:perspectives},   the Pollard distribution $P_\alpha$  
is known as  the Mittag-Leffler distribution in probabilistic literature.
For completeness, we briefly discuss a different generalisation  of  $P_\alpha$ 
that features extensively in  such literature.
It is  known as the  generalised Mittag-Leffler distribution $P_{\alpha,\theta}$
(Pitman~\cite{Pitman_CSP}, p70 (3.27)),  also   denoted by ${\rm ML(\alpha,\theta})$
(Goldschmidt and Haas~\cite{GoldschmidtHaas}, Ho et al.~\cite{HoJamesLau}).

Despite its name, 
$P_{\alpha,\theta}(t)$ is different from
the two-parameter Pollard distribution $P_{\alpha,\beta}(t)$  discussed above, 
whose Laplace transform is  the Mittag-Leffler function $E_{\alpha,\beta}(-x)$.
Janson~\cite{Janson} showed that $P_{\alpha,\theta}$  may be  constructed  
as a limiting distribution of a   P{\' o}lya urn scheme. 
It is also intimately linked to a concept known as  `polynomial tilting'. 
For some parameter $\theta$, 
$f_{\alpha,\theta} (x)\propto  x^{-\theta}f_\alpha(x)$ 
 is said to be a polynomially  tilted variant of $f_\alpha(x)$
({\it e.g.}\ Arbel et al.~\cite{Arbel}, Devroye~\cite{Devroye}, James~\cite{James_Lamperti}).
Here, we consider the polynomially tilted  density 
 $f_{\alpha,\theta} (x\vert t)\propto x^{-\theta}f_\alpha (x\vert t)$
conditioned on a  scale factor $t>0$. 
Normalisation gives 
\begin{align}
 f_{\alpha,\theta}(x\vert t)= \frac{\Gamma(\theta+1)}{\Gamma(\theta/\alpha+1)} t^{\theta/\alpha} \,x^{-\theta}f_\alpha (x\vert t)
\label{eq:tiltedstablenormfactor}
\end{align}
so that $f_{\alpha,\theta}(x\vert t)$ is defined for $\theta/\alpha+1>0$, or $\theta>-\alpha$.
We then consider a two-parameter function  $h_{\alpha,\theta} (x\vert\lambda)$ defined by:
\begin{align}
  \alpha\, h_{\alpha,\theta} (x\vert\lambda) &= x \int_0^\infty f_{\alpha,\theta} (x\vert t) \, t^{-1} e^{-\lambda t} \,dt 
\label{eq:ah_at} \\
   &= \frac{\Gamma(\theta+1)}{\Gamma(\theta/\alpha+1)} \, x^{1-\theta}
    \int_0^\infty f_{\alpha}(x\vert t)\, t^{\theta/\alpha-1}\, e^{-\lambda t} \, dt  \nonumber \\
u=x^{-\alpha}t:\quad
 h_{\alpha,\theta} (x\vert\lambda) 
    &= \int_0^\infty e^{-\lambda x^\alpha u} \, dP_{\alpha,\theta}(u) 
 \label{eq:h_at} \\
{\rm where}\quad    
   P_{\alpha,\theta}(t)
   &= \frac{\Gamma(\theta+1)}{\Gamma(\theta/\alpha+1)} \,  \frac{1}{\alpha}
    \int_0^t  f_\alpha(u^{-1/\alpha}) \, u^{(\theta-1)/\alpha-1} \, du  
\label{eq:P_at1} \\
{\rm or}\quad    
   dP_{\alpha,\theta}(t)
   &= \frac{\Gamma(\theta+1)}{\Gamma(\theta/\alpha+1)}\, t^{\theta/\alpha} \,dP_\alpha(t)
\label{eq:P_at2} 
\end{align}
It is clear from~(\ref{eq:h_at}) that $h_{\alpha,\theta} (x\vert\lambda)$ may be written as  $h_{\alpha,\theta} (\lambda x^\alpha)$. 
It  follows that:
\begin{enumerate}
\item $h_{\alpha,\theta} (x)$  is completely monotone
\item $\theta=0$: $P_{\alpha,0}(t)=P_\alpha(t) \implies h_{\alpha,0}(x)=E_\alpha(-x)$,
as directly apparent from comparing (\ref{eq:ML_intrep0}) and (\ref{eq:ah_at}).
\item $h_{\alpha,\theta} (\eta)$  is completely monotone where $\eta$ is a Bernstein function 
as discussed in Section~\ref{sec:consequence}.
In particular,  $h_{\alpha,\theta} (\lambda x^\alpha)$ is completely monotone and thus
expressible as the Laplace  transform of a corresponding distribution $Q_{\alpha,\theta} (t\vert\lambda)$ 
(distinct from $Q_{\alpha,\beta} (t\vert\lambda)$ discussed in Section~\ref{sec:consequence}).
\end{enumerate}
We are not aware of a  representation of $h_{\alpha,\theta}$ other than that generated by $P_{\alpha,\theta}$ in~(\ref{eq:h_at}).
By comparison,  the two-parameter Mittag-Leffler function $E_{\alpha,\beta}$ has a well-established 
infinite series representation~(\ref{eq:ML2parseries}),
in addition to the representation~(\ref{eq:ML2par}) generated by the  
two-parameter Pollard distribution $P_{\alpha,\beta}$.

\section{Discussion}
\label{sec:discussion}

The integral representation~(\ref{eq:ML3par}) of  $E^\gamma_{\alpha,\beta}(-\lambda x^\alpha)$ in Theorem~\ref{thm:ML3par},
arising from the limit~(\ref{eq:limitgenmarginaldensity}),
contains the L\'{e}vy measure $t^{-1}e^{-\lambda t}dt$ of the infinitely divisible gamma distribution.  
There is indeed an   intimate relationship between completely monotone functions and the theory of 
infinitely divisible  distributions on the nonnegative half-line $\mathbb{R}_{+}=[0,\infty)$ 
(Feller~\cite{Feller2} (XIII.4,~XIII.7), Steutel and van Harn~\cite{SteutelvanHarn} (III)).
Sato~\cite{Sato}  considers infinitely divisible distributions on $\mathbb{R}^d$,
but the deliberate restriction to $\mathbb{R}_{+}$ makes for simpler discussion 
and relates directly to the core concept of complete monotonicity that is of interest here.
There is also an intimate  link to the generalised gamma convolutions studied by Bondesson~\cite{Bondesson}.

The   limit~(\ref{eq:limitgenmarginaldensity}) of Theorem~\ref{thm:ML3par} 
is an instance of a limit rule to generate the 
L\'{e}vy measure of an infinitely divisible distribution given in Steutel and van Harn~\cite{SteutelvanHarn}~(III(4.7)) and
 Sato~\cite{Sato} (Corollary~8.9 restricted to $\mathbb{R}_{+}$ rather than $\mathbb{R}^d$).
Barndorff-Nielsen and Hubalek~\cite{BarndorffHubalek} also cite Sato's Corollary.

Further exploration  using the probabilistic  machinery of this paper
possibly includes the  explicit  determination of  the three-parameter distribution 
$Q^\gamma_{\alpha,\beta}(t\vert\lambda)$, whose Laplace transform is 
$E^\gamma_{\alpha,\beta}(-\lambda x^\alpha)$, as per~(\ref{eq:Qabg}).

\section{Conclusion}
\label{sec:conclusion}

We have presented a probabilistic  derivation of  the complete monotonicity of the three-parameter Mittag-Leffler function 
(also known as the Prabhakar function) by expressing it as the Laplace transform of a distribution that 
we referred to as the three-parameter Pollard distribution.
This is a generalisation of a result due to Pollard for the one-parameter case.


 \bibliography{MittagLeffler.bib}{}
\bibliographystyle{plain} 
\end{document}